\documentclass[letterpaper, 10 pt, conference]{ieeeconf}  

\IEEEoverridecommandlockouts                              

\usepackage{graphics} 
\usepackage{epsfig} 
\usepackage{mathptmx} 
\usepackage{times} 
\usepackage{amsmath} 
\usepackage{amssymb}  
\usepackage{xcolor}
\usepackage{subcaption}
\usepackage{float}
\usepackage{hyperref}
\hypersetup{colorlinks,allcolors=black}

\newcommand\numberthis{\addtocounter{equation}{1}\tag{\theequation}}

\newtheorem{theorem}{Theorem}
\newtheorem{lemma}{Lemma}
\newtheorem{definition}{Definition}
\newtheorem{remark}{Remark}
\newtheorem{assumption}{Assumption}
\newtheorem{proposition}{Proposition}
\newtheorem{corollary}{Corollary}

\title{\LARGE \bf
Indirect Adaptive Optimal Control in the Presence of Input Saturation*
}

\author{Sunbochen Tang and Anuradha M. Annaswamy
\thanks{*This work was supported by Ford-MIT Alliance.}
\thanks{Sunbochen Tang (tangsun@mit.edu) and Anuradha M. Annaswamy (aanna@mit.edu) are with the Department of Mechanical Engineering, Massachusetts Institute of Technology, Cambridge, MA 02139, USA.
}}

\begin{document}

\maketitle

\begin{abstract}
    In this paper, we propose a combined Magnitude Saturated Adaptive Control (MSAC)-Model Predictive Control (MPC) approach to linear quadratic tracking optimal control problems with parametric uncertainties and input saturation. The proposed MSAC-MPC approach first focuses on a stable solution and parameter estimation, and switches to MPC when parameter learning is accomplished. We show that the MSAC, based on a high-order tuner, leads to parameter convergence to true values while providing stability guarantees. We also show that after switching to MPC, the optimality gap is well-defined and proportional to the parameter estimation error. We demonstrate the effectiveness of the proposed MSAC-MPC algorithm through a numerical example based on a linear second-order, two input, unstable system.
\end{abstract}

\section{Introduction}
Adaptive control is developed to address the presence of parametric uncertainties. The term ``adaptive'' suggests that an adaptive controller can modify its behavior in response to sudden changes in the dynamics of the process \cite{aastrom2013adaptive}, with emphasis on a real-time solution that achieves typical control objectives including tracking a reference signal and stabilization. Therefore, the performance is evaluated with present and instantaneous properties, e.g., stability guarantees. In general, future behaviors such as optimality of the system trajectory are not considered in the goals of adaptive control systems.

Optimal control methods, in contrast, focus on finding control strategies to minimize a specific cost function that the system is expected to incur over a future time horizon. Other control goals such as stabilization and reference tracking are often incorporated through specific cost function designs or constrained optimization formulations. To achieve cost minimization, optimal control strategies rely on having an accurate model, which requires knowledge of both the model structure and parameters within.

When parametric uncertainties occur, an accurate model is no longer available, and hence the problem becomes more complex, as one has to determine a controller that both adapt to uncertainties and achieve optimality goals. We propose a solution to this complex problem which combines Magnitude Saturated Adaptive Control (MSAC) and Model Predictive Control (MPC) approaches. The proposed MSAC-MPC controller first focuses on fast adaptation by leveraging high-order tuner-based adaptation laws, which guarantees stability of the closed-loop system, and ensures parameter learning in the presence of persistent excitation \cite{gaudio2020class}. Once parameter learning is accomplished, our proposed controller shifts its focus to the optimality objective, which utilizes a MPC structure to minimize cost using the estimated model learnt during adaptation. In addition, the MSAC component in our proposed approach incorporates control input magnitude saturation, which is equivalent to input constraints commonly seen in optimal control problems. All discussions are limited to linear time-invariant systems.

In the literature, approaches to solve adaptive optimal control problems can be broadly categorized into direct and indirect adaptive optimal control methods (see \cite{narendra2012stable} for definitions). As introduced in \cite{sutton1992reinforcement}, indirect adaptive optimal control methods distinguish from their direct counterpart by explicitly learning the unknown parameters. Indirect adaptive optimal control approaches often utilize a robust optimal control framework. As an example, the Adaptive MPC proposed in \cite{adetola2009adaptive} iteratively updates estimates of the uncertain parameters and their error bounds, and optimizes with respect to the worst case using a robust MPC design \cite{bemporad2003min}. Although this approach provides proof of stability guarantees and achieves optimality goals, it accomplishes these goals by solving a min-max optimization problem to ensure robustness over all possible uncertain parameter values in a compact set, which is computationally very expensive, thus less suitable for real-time decision making tasks. In comparison, our proposed approach focuses on real-time control and parameter learning during adaptation, therefore, reduces the estimation error of uncertain parameters rapidly, and avoids the need for invoking robust MPC.

Direct adaptive optimal control methods, which refer to methods that do not learn uncertain parameters explicitly, are mainly developed based on online approximate dynamic programming (ADP) algorithms such as value iteration and policy iteration \cite{bertsekas2008approximate}. As an example, the computational adaptive optimal control method proposed in \cite{jiang2012computational} for linear systems with unknown matrices uses policy iteration to iteratively solve the Riccati equation and consequently find optimal control policy, instead of identifying system dynamics. Similar approaches are developed in \cite{jiang2015global}, \cite{vrabie2008adaptive}, and \cite{bhasin2013novel}. Although these online algorithms are shown to be very effective and provide stability guarantees, most of them require an initial stable control policy, which is often unavailable especially in the case where parametric uncertainties occur online. Our proposed MSAC approach avoids this difficulty and ensures that the proposed real-time controller will guarantee boundedness and a small tracking error, without any requirement of an initial controller that guarantees closed-loop stability.

The main contribution of this paper is that the proposed MSAC-MPC approach provides real-time stable solution and parameter estimation during MSAC adaptation, and achieves near-optimal solution using MPC after parameter learning. Further, our approach is applicable to both finite-horizon and infinite-horizon optimal control problems. 

The remainder of this paper is organized as follows. In Section II, we formulate the adaptive optimal control problem. In Section III, we introduce the linear quadratic tracking problem and derive a structure of optimal solutions under input constraints. In Section IV, we propose the MSAC-MPC controller and present theoretical results including stability guarantees, parameter learning, and optimality analysis. In Section V, we demonstrate the proposed approach in a numerical example. The paper is concluded in Section VI.

\section{Problem Formulation}
Consider a linear system with parametric uncertainties,
\begin{equation}\label{equ:plant}
    \dot{x}_p = A_p x_p + B_p \Lambda (B_{sat}(u) )
\end{equation}
where $x_p \in \mathbb{R}^{n_x} $ is the state vector that is assumed to be measurable, $u \in \mathbb{R}^{n_u}$ is the control input vector. The unknown parameters are $A_p \in \mathbb{R}^{n_x \times n_x}$, $\Lambda \in \mathbb{R}^{n_u \times n_u}$, and we assume $\Lambda$ is a diagonal matrix with all entries positive. The function $B_{sat}(u)$ represents control input saturation by a ball centered at origin with a fixed radius $u_{max} \in \mathbb{R}$, i.e.,
\begin{equation}\label{equ:input_sat}
    B_{sat}(u) = 
    \begin{cases}
        u, &\|u\| \leq u_{max}\\
        \frac{\|u\|}{u_{max}}u, &\|u\| > u_{max}
    \end{cases}
\end{equation}

The objective is to compute an optimal control law for the system in \eqref{equ:plant} such that the following quadratic cost function is minimized over a prediction horizon of length $T = t_1 - t_0$,
\begin{equation}\label{equ:cost_function}
    J = h(x_p(t_1)) + \int_{t_0}^{t_1} l(x_p, u, \tau) d\tau
\end{equation}
where $h(x) = (x-x_d(t_1))^T Q_f (x- x_d(t_1)), Q_f = Q_f^T \succeq 0$ is the terminal cost matrix, and $l(x, u, \tau) = (x(\tau) - x_d(\tau))^TQ(x(\tau) - x_d(\tau)) + u(\tau)^T R u(\tau) , Q = Q^T  \succeq 0, R = R^T \succ 0$ is the stage cost matrices. Note that the stage cost is penalizing a tracking error, $e_p = x_p - x_d$, which forces the system to track an exogenous signal $x_d(\tau)$. 

In the optimal control framework, the input magnitude saturation in \eqref{equ:input_sat} is equivalent to input constraints. Therefore, the overall problem we are addressing in this paper can be formulated as a linear quadratic tracking (LQT) problem with parametric uncertainties in system dynamics,
\begin{subequations}
    \begin{align}
        \min_{u(\tau)} J &= h(x_p(t_1)) + \int_{t_0}^{t_1} l(x_p, u, \tau) d\tau \label{equ:LQT_uncertain_1}\\
        \text{subject to } & x(t_0) = x_0 \\
        &\dot{x}_p(\tau) = A_p x_p(\tau) + B_p \Lambda (u(\tau) ) \quad \forall \tau \in [t_0, t_1] \label{equ:LQT_uncertain_model}\\
        & u(\tau)  \in U ,\quad \forall \tau \in [t_0, t_1]\\
        & U = \{u\in \mathbb{R}^{n_u}: \|u\| \leq u_{max}\} \label{equ:LQT_uncertain_2}
    \end{align}
\end{subequations}
Two assumptions are made to simplify the problem:
\begin{assumption}\label{assum:R_diagonal}
    $R$ is a diagonal matrix with identical entries, i.e., $\exists R_u > 0$ such that $R = R_u I_{n_u\times n_u}$.
\end{assumption}

\begin{assumption}\label{assum:known_xd}
    The exogenous signal $x_d(t)$ that is desired to be tracked and its derivative $\dot{x}_d(t)$ are known.
\end{assumption}

The main difficulty is that directly solving the optimization problem when $A_p$ and $\Lambda$ are unknown in the dynamics constraints is intractable. In addition, LQT problems with input constraints generally do not have an explicit solution because the solution involves exogenous signal $x_d(t)$, unlike linear quadratic regulation (LQR) problems where stabilization is the goal. 

The solution we propose, termed MSAC-MPC, proceeds thus. First, we propose a  magnitude-saturation constrained adaptive controller (MSAC), where the controller guarantees bounded solutions and persistent excitation of the underlying regressor guarantees learning of the parameters. The resulting parameter estimates are then used to switch to an MPC controller, and the overall cost is evaluated in terms of the residual parameter error after a finite time. Numerical examples are provided to illustrate the nature of the overall MSAC-MPC controller.

\section{Linear Quadratic Tracking Control}
We first address the LQT problem in \eqref{equ:LQT_uncertain_1}-\eqref{equ:LQT_uncertain_2} when $A_p$ and $\Lambda$ are known. When there is no control input constraint, i.e., $U = \mathbb{R}^{n_u}$, the unconstrained optimal controller $u_{uc}^*(\tau)$ assumes a linear feedback form, and the corresponding optimal cost-to-go function is quadratic in $x(\tau)$, \cite{underactuated}
\begin{subequations}\label{equ:uc_opt_u}
    \begin{align}
        u_{uc}^*(\tau) &= -R^{-1} (B_p \Lambda)^T(S_1^{uc}(\tau) + S_2^{uc} (\tau) x(\tau))\\
        V_{uc}^*(x, \tau) &= x(\tau)^T S_2^{uc}(\tau) x(\tau) + 2x(\tau)^T S_1^{uc}(\tau) + S_0^{uc}(\tau)
    \end{align}
\end{subequations}

\subsection{Linear feedback control}
In this subsection, we show the cost-to-go function associated with \eqref{equ:LQT_uncertain_1}-\eqref{equ:LQT_uncertain_2} under any linear feedback control is quadratic. We also provide a bound on the cost-to-go function in terms of error bounds on the control parameters. Following Lemma 1 in \cite{vrabie2008adaptive}, a cost-to-go function $V^{\pi}$ associated with a fixed feedback control policy $u(\tau) = \pi(x(\tau))$ can be found by solving the following equation,
\begin{align}\label{equ:HJB_policy_eval}
    l(x, u, \tau) + \frac{\partial V^{\pi}}{\partial x} (A_p x + B_p \Lambda u) + \frac{\partial V^{\pi}}{\partial \tau} = 0
\end{align}

In the case where $\pi(x(\tau))$ is a linear feedback control policy, by solving \eqref{equ:HJB_policy_eval}, we have the following proposition.
\begin{proposition}\label{prop: policy_evaluation}
    Under any linear feedback control law $u(\tau) = \pi(x, \tau) = K_1^{\pi}(\tau) x(\tau) + K_0^{\pi}(\tau)$, the cost-to-go function is quadratic, $V^{\pi} = x^T S_2^{\pi} x + 2x^T S_1^{\pi} + S_0^{\pi}$.
    \begin{subequations}
        \begin{align*}
            \dot{S}_2^{\pi} &= - Q - S_2^{\pi} (A_p+B_p\Lambda K_1^{\pi}) - (A_p+ B_p \Lambda K_1^{\pi})^T {S_2^{\pi}}^T\\
            &- {K_1^{\pi}}^T R K_1^{\pi} \numberthis{}\label{equ:policy_eval_1}\\
            \dot{S}_1^{\pi} &= Q x_d - [S_2^{\pi} K_0^{\pi} + (A_p+B_p\Lambda K_1^{\pi})^T {S_1^{\pi}}^T + {K_1^{\pi}}^T R K_0^{\pi}]\numberthis{}\\
            \dot{S}_0^{\pi} &= - x_d^T Q x_d - 2S_1^{\pi} K_0^{\pi} - {K_0^{\pi}}^T R K_0^{\pi}\numberthis{} \label{equ:policy_eval_2}
        \end{align*}
    \end{subequations}
    with boundary conditions,
    \begin{equation*}
        S_2^{\pi}(t_1) = Q_f, S_1^{\pi}(t_1) = -Q_f x_d(t_1), S_0^{\pi}(t_1) = x_d^T Q_f x_d
    \end{equation*}
\end{proposition}
\begin{remark}
    When $K_1^{\pi} = -R^{-1} (B_p \Lambda)^T S_2^{uc}$ and $K_0^{\pi} = -R^{-1} B^T S_1^{uc}$, the unconstrained optimal control and optimal cost-to-go function are recovered, i.e., $u = u^*_{uc}$ and $V^{\pi} = V_{uc}^*$.
\end{remark}

Based on the results in Proposition~\ref{prop: policy_evaluation}, for two linear feedback control laws with similar gain matrices, the resulting cost-to-go functions will also be close in their values.
\begin{proposition}\label{prop: dist_LQ_HJB}
    For any two linear feedback laws $\pi_i(x, \tau) = K_1^{\pi_i}(\tau) x(\tau) + K_0^{\pi_i}(\tau)$, $i = 1, 2$, if there exists $\delta >0$, such that $\max_j\{\sup_{\tau \in [t_0, t_1]} \|K_j^{\pi_1} - K_j^{\pi_2}\|\} \leq \delta$, then the corresponding cost-to-go functions, defined by \eqref{equ:policy_eval_1}-\eqref{equ:policy_eval_2}, $V^{\pi_1}, V^{\pi_2}$, satisfy $|V^{\pi_1} - V^{\pi_2}| \leq O(\delta^2)$.
\end{proposition}
Quadratic coefficient matrices, $S_j^{\pi}$, are defined by differential equations in \eqref{equ:policy_eval_1}-\eqref{equ:policy_eval_2}. It is not difficult to see that $S_2^{\pi}$ is a quadratic function of $K_1^{\pi}$, $S_1^{\pi}$ is a quadratic function of $K_1^{\pi}, K_0^{\pi}$, and $S_0^{\pi}$ is a quadratic function of $K_0^{\pi}$. In addition, $V^{\pi}$ is a linear function of $S_2^{\pi}, S_1^{\pi}, S_0^{\pi}$, therefore, a quadratic function of $K_1^{\pi}, K_0^{\pi}$, which implies,
\begin{equation*}
    |V^{\pi_1} - V^{\pi_2}| = O(\max_{j = 0, 1}\{\sup_{\tau \in [t_0, t_1]} \|K_j^{\pi_1} - K_j^{\pi_2}\|\}) = O(\delta^2)
\end{equation*}

\subsection{Input Constrained LQT}
Now we consider the LQT problem with input constraints described in \eqref{equ:LQT_uncertain_1}-\eqref{equ:LQT_uncertain_2} with $U \subsetneqq \mathbb{R}^{n_u}$. In the following theorem, we show that the constrained optimal control input $u^*$ is the point in admissible set $U$ closest to unconstrained optimal control $u_{uc}^*$ under a distance metric defined by $R$, i.e. $d_R(x, y) = (x - y)^T R (x - y)$.
\begin{theorem}
    If $u^*$ is an optimal solution to \eqref{equ:LQT_uncertain_1}-\eqref{equ:LQT_uncertain_2}, then $u^*$ satisfies $u^* = argmin_{u\in U} (u - u_{uc}^*)^T R (u - u_{uc}^*)$.
\end{theorem}
\begin{proof}
    Let $\mathcal{H}$ be the Hamiltonian of \eqref{equ:LQT_uncertain_1}-\eqref{equ:LQT_uncertain_2}, 
    \begin{equation*}
        \mathcal{H} = p_0 l(x(\tau), u(\tau), \tau) + p(\tau)^T (A_p x(\tau) + B_p \Lambda u(\tau))
    \end{equation*}
    where $p_0 \in {0, 1}$ is a binary variable constant in $\tau$, $p(\tau) \in \mathbb{R}^{n_x}$ is the Lagrange multiplier. Since there is no additional equality constraint other than system dynamics, $p_0 = 1$ \cite{mangasarian1966sufficient}.
    
    Using Pontryagin's Maximum Principle (PMP), by minimizing $\mathcal{H}$ with resepct to $u$ subject to input constraint $u \in U$, the optimal control input under constraint $u^*$ is
    \begin{align*}
        u^* = argmin_{u\in U} {u(\tau)^T R u(\tau) + p(\tau)^T B_p \Lambda u(\tau)}
    \end{align*}
    Note that for the reduced quadratic cost, unconstrained optimal solution is $u_{uc}^* = -\frac{1}{2}R^{-1}(B_p\Lambda)^T p$, which satisfies
    \begin{align*}
        u^T R u + p^T B_p \Lambda u &= (u - u_{uc}^*)^T R (u - u_{uc}^*) - {u_{uc}^*}^T R u_{uc}^*
    \end{align*}
    Since $u_{uc}^*$ does not depend on $u$ explicitly, we have
    \begin{equation*}
        u^* = argmin_{u\in U} (u - u_{uc}^*)^T R (u - u_{uc}^*)
    \end{equation*}
    
    The Lagrange multipliers $p(\tau)$ are defined by the following differential equations and boundary conditions, which are the adjoint equations and transversality conditions in PMP,
    \begin{align*}
        \dot{p}(\tau) &= -\nabla_x \mathcal{H} = p(\tau)^T A_p - 2 Q(x(\tau) - x_d(\tau))\\
        p(t_1)  &= 2Q(x(t_1) - x_d(t_1))\\
        0 &= l(t_1) + p(t_1)^T(A_p x(t_1) + B_p \Lambda u(t_1))
    \end{align*}
    
\end{proof}
\begin{corollary}\label{coro: ball_cons_LQT}
    Under Assumption~\ref{assum:R_diagonal}, and the admissible set $U$ is a ball around origin, i.e. $U = \{u \in \mathbb{R}^{n_u}: \|u\|^2 \leq u_{max}\}, u_{max} > 0$, $u^*$ can be written as $u^* = u_{uc}^* \frac{u_{max}}{\|u_{uc}^*\|}$, which is a linear feedback controller.
\end{corollary}
\begin{proof}
    By Theorem 1, when $R = R_u I_{n_u \times n_u}$, 
    \begin{equation*}
        u^* = argmin_{u\in U} R_u \|u - u_{uc}^*\|^2 = argmin_{u\in U} \|u - u_{uc}^*\|^2
    \end{equation*}
    If $U$ is a ball of radius $u_{max}$ around the origin, then the closest point to $u^*$ is the projection of $u^*$ on the outside sphere, therefore, $u^* = u_{uc}^* \frac{u_{max}}{\|u_{uc}^*\|}$.
\end{proof}

\section{MSAC-MPC: Adapt, Learn, Optimize}
In the previous section, we have established that the optimal solution to a input-constrained LQT problem with known system matrices \eqref{equ:LQT_uncertain_1}-\eqref{equ:LQT_uncertain_2} is a linear feedback controller, with the assumption that the parameters were known. It should be noted that these parameters were directly used in the LQT control design. Now we address the input-constrained LQT problem with parametric uncertainties. As illustrated in Figure~\ref{fig:flow_chart}, the proposed MSAC-MPC controller first adapts over $[0, T_{adap}]$ and then switches to an MPC controller. 

The main difficulty in extending the approach described in Section III lies in computing optimal control while parametric uncertainties are present in the system dynamics. Adaptive MPC approaches, e.g., in \cite{adetola2009adaptive}, propose to estimate the unknown parameters based on state measurements while computing optimal control using currently estimated parameter values. While these approaches provide stability guarantees through robust MPC designs, the resulting min-max optimization problems are computationally quite burdensome. From a real-time control perspective, when computation power is limited and control input needs to be determined within a short time, the proposed MSAC-MPC controller may prove to be attractive. During $[t_0, t_0 + T_{adap}]$, our controller focuses primarily on a stable solution and on learning the unknown parameters. After this finite time, we show that one can switch to MPC with a well-defined optimality gap that is proportional to the parameter error that is present for $t\geq t_0 + T_{adap}$.

\begin{figure}[H]
    \centering
    \includegraphics[width=0.5\textwidth]{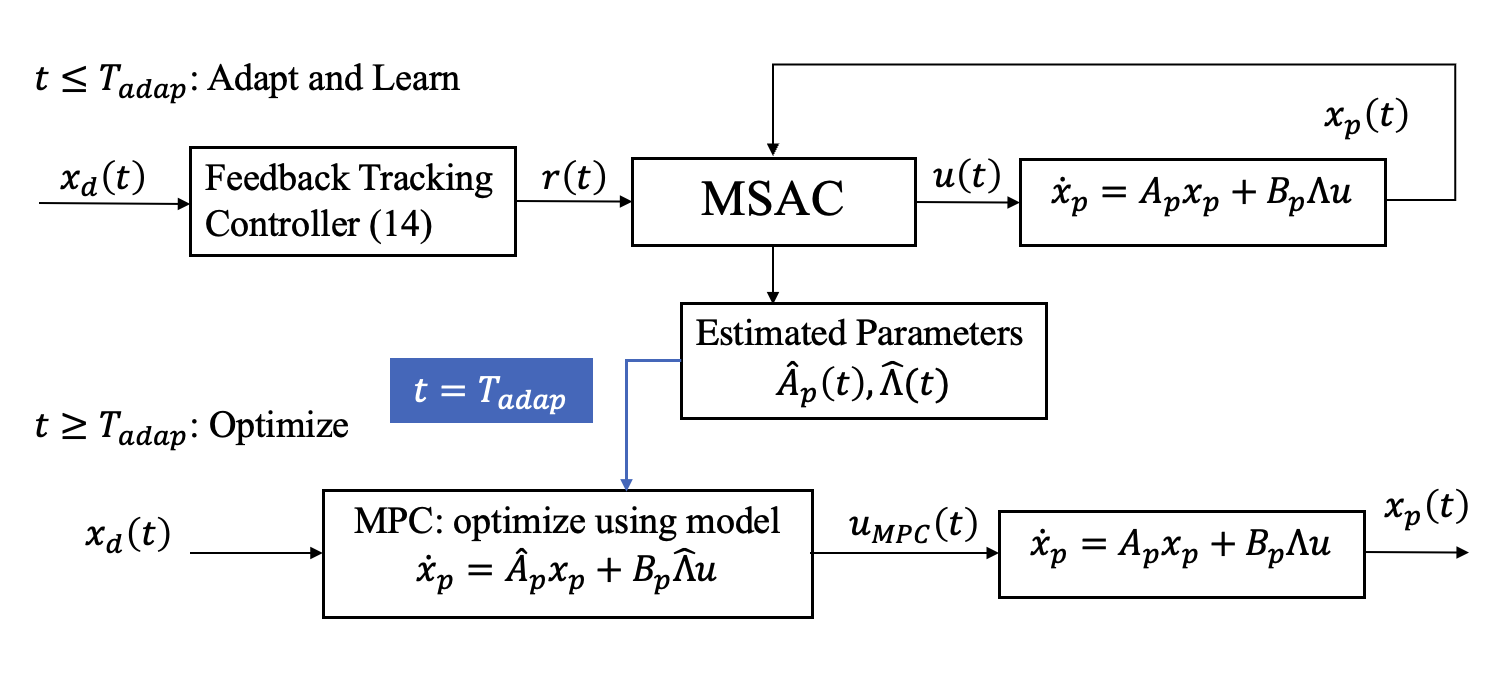}
    \caption{A diagram of MSAC-MPC Algorithm ($t_0$ is assumed to be zero).}
    \label{fig:flow_chart}
\end{figure}

\subsection{MSAC: Adapt and Learn}
The MSAC controller in this paper follows the adaptive control structure proposed in \cite{schwager2005direct} for multi-input systems, but replaces the parameter adaptation by a high-order tuner in \cite{gaudio2020class}. The motivation behind such a combination is that we explicitly accommodate input constraints and guarantee boundedness as in \cite{schwager2005direct} while ensuring that the speed of convergence is fast when compared to \cite{schwager2005direct}, as in \cite{gaudio2020class}.

\subsubsection{Adaptive control design}
As in all adaptive control designs \cite{narendra2012stable}, a known reference system is chosen as,
\begin{equation}\label{equ:ref_sys}
    \dot{x}_m = A_m x_m + B_m r
\end{equation}
where $x_m \in \mathbb{R}^{n_x}$, $r \in \mathbb{R}^{n_u}$ is a bounded reference input satisfying $\|r\| \leq r_{max}$, and $A_m$ is a Hurwitz matrix, and $B_m$ is full column rank.

The magnitude saturated multi-input model reference adaptive controllrer (MSAC) is defined as (see \cite{schwager2005direct} for details),
\begin{align}\label{equ:AC}
    u = \widehat{\Theta} \Phi,\quad \widehat{\Theta} = [\hat{K}_x, \hat{K}_r], \quad \Phi = [x_p^T, r^T]^T
\end{align}

We assume there exists an ideal matrix $K_x \in \mathbb{R}^{n_u\times n_x}$ and an ideal vector $K_r \in \mathbb{R}^{n_u \times n_u}$ corresponding to the unknown parameters in \eqref{equ:plant}, which satisfy
\begin{align*}
    A_p + B_p \Lambda K_x = A_m, \quad B_p\Lambda K_r &= B_m
\end{align*}

The parameter estimation errors are defined as,
\begin{equation*}
    \Tilde{K}_x = \hat{K}_x - K_x, \quad \Tilde{K}_r = \hat{K}_r - K_r, \quad \widetilde{\Theta} = [\Tilde{K}_x, \Tilde{K}_r]
\end{equation*}

Let $\Delta u = B_{sat}(u) - u$. The closed-loop dynamics is,
\begin{equation}\label{equ:CL_plant}
    \dot{x}_p = (A_p + B_p \Lambda K_x) x_p + B_p \Lambda K_r r + B_p\Lambda (\widetilde{\Theta} \Phi + \Delta u)
\end{equation}

By subtracting \eqref{equ:ref_sys} from \eqref{equ:plant} and defining error $e = x_p - x_m$, the closed-loop error dynamics is obtained as,
\begin{equation}\label{equ:e_dyn}
    \dot{e} = A_m e + B_p\Lambda (\widetilde{\Theta} \Phi + \Delta u)
\end{equation}

To address the nonlinear disturbance $\Delta u$ introduced by the input magnitude saturation, we introduce an auxiliary error $e_{\Delta}$ generated by the following dynamics,
\begin{align*}
    \dot{e}_{\Delta} = A_m e_{\Delta} + B_p diag(\hat{\lambda})\Delta u
\end{align*}
where $\hat{\lambda} \in \mathbb{R}^{n_u}$ is a vector that estimates the unknown diagonal entries of $\Lambda$. We also denote the diagonal vector of $\Lambda$ as $\lambda$, and define $\Tilde{\lambda} = \hat{\lambda} - \lambda$. 

Define $e_u = e - e_{\Delta}$, the augmented error $e_u$ is used to establish the stability results, whose error dynamics is,
\begin{align}\label{equ:e_u_dyn}
    \dot{e}_u = A_m e_u + B_p[\Lambda \widetilde{\Theta}, diag(\Tilde{\lambda})] [\Phi^T \Delta u^T]^T
\end{align}
Let $\Theta_a = [{\Theta}, \Lambda]$, $\widehat{\Theta}_a = [\widehat{\Theta}, diag(\hat{\lambda})]$, $\Phi_a = [\Phi^T, -\Delta_u^T]^T$. The High-order Tuner-based Adaptation laws are chosen as,
\begin{subequations}\label{equ:HT_adap}
    \begin{align}
        \dot{\Xi}_a &= -\gamma B_p^T P e_u \Phi_a^T \label{equ:HT_1}\\
        \dot{\widehat{\Theta}}_a &= -\beta(\widehat{\Theta}_a - \Xi_a) \mathcal{N}_t\\
        \mathcal{N}_t &= 1 + \mu \Phi_a^T \Phi_a, \quad \mu \geq \frac{2\gamma}{\beta}\|PB_p\|^2_F \label{equ:HT_2}
    \end{align}
\end{subequations}
where $P$ is the solution to Lyapunov equation, $A_m^T P + PA_m = -Q$ and $Q$ is a positive definite matrix.


The goal of the adaptive controller is to show that the adaptation laws in \eqref{equ:HT_1}-\eqref{equ:HT_2} guarantee that the tracking error $e$ remains small and that all signals in the closed-loop system are bounded. Once this is accomplished, persistent excitation arguments are to be used to ensure that $\widehat{\Theta}_a$ converges to the true value. We now link the convergence of the error $e$ to the tracking of the desired signal $x_d$ by specifying the reference input $r(t)$ in the reference model as follows.

Using Assumption~\ref{assum:known_xd}, we choose the reference input $r$ as
\begin{equation}\label{equ:feedback_r}
    r = (B_m^T B_m)^{-1} B_m^T(- A_m x_d + \dot{x}_d)
\end{equation}
Let $e_m = x_m - x_d$. Since the error dynamics is $\dot{e}_m = A_m e_m$, and $A_m$ is Hurwitz, we have that $\|e_m(t)\| \to 0$ as $t\to \infty$.

Now that we have a complete MSAC design described in \eqref{equ:AC}, \eqref{equ:e_u_dyn}, \eqref{equ:HT_1}-\eqref{equ:HT_2}, and \eqref{equ:feedback_r}, we are ready to state stability results and parameter learning results.

\subsubsection{Stability results}
Let $\widetilde{\Theta}_a = \widehat{\Theta}_a - \Theta_a$. Consider a Lyapunov function candidate, 
\begin{align*}
    V &= e_u^T P e_u + \frac{1}{\gamma} Tr\left[ (\Xi_a - \Theta_a)^T \Lambda_a^T (\Xi_a - \Theta_a) \right]\\
    &+ \frac{1}{\gamma} Tr\left[ (\widehat{\Theta}_a - \Xi_a)^T \Lambda_a^T (\widehat{\Theta}_a - \Xi_a) \right] \numberthis{} \label{equ:Lyapunov}
\end{align*}
where $\Lambda_a = \begin{bmatrix}
    \Lambda & 0 \\ 0 & I_{n_u \times n_u} 
\end{bmatrix}$
is used for compact notation.

It follows that 
\begin{align*}
    \dot{V} &= e_u^T (P A_m + A_m^T P) e_u + 2 e_u^T P B_p [\Lambda \widetilde{\Theta}\Phi - diag(\Tilde{\lambda})\Delta u]\\
    &-Tr[(\Xi_a - \Theta_a)^T(\Lambda_a + \Lambda_a^T) B_p^T P e_u \Phi_a^T]\\
    &-\frac{\beta}{\gamma} Tr\left[ (\widehat{\Theta}_a - \Xi_a)^T(\Lambda_a + \Lambda_a^T) (\widehat{\Theta}_a - \Xi_a) \right]\mathcal{N}_t\\
    &+ Tr\left[ (\widehat{\Theta}_a - \Xi_a)^T (\Lambda_a + \Lambda_a^T) B_p^T P e_u \Phi_a^T \right]\\
    &= -e_u^T Q e_u + 2 e_u^T P B_p [\Lambda \widetilde{\Theta}\Phi - diag(\Tilde{\lambda})\Delta u]\\
    &-Tr\left[ (\Xi - \Theta)^T(\Lambda + \Lambda^T) B_p^T P e_u \Phi^T\right]\\
    &+ 2Tr\left[(diag(\lambda_\epsilon) - diag(\hat{\lambda}))^T B_p^T P e_u \Delta u^T \right]\\
    &- \frac{\beta}{\gamma} \mathcal{N}_t [ Tr[(\widehat{\Theta}-\Xi)^T(\Lambda + \Lambda^T)(\widehat{\Theta}-\Xi)]\\
    &+ 2Tr(diag(\hat{\lambda}) - diag(\lambda_{\epsilon}))^T(diag(\hat{\lambda}) - diag(\lambda_{\epsilon})) ]\\
    &+ Tr[(\widehat{\Theta}-\Xi)^T(\Lambda + \Lambda^T)B_p P e_u \Phi^T]\\
    &+ 2 Tr[(diag(\hat{\lambda}) - diag(\lambda_{\epsilon}))^TB_p^T P e_u (-\Delta u)^T]\\
    &= -e_u^T Q e_u + 4e_u^TPB_p\Lambda (\widehat{\Theta} - \Xi)\Phi \\
    &-\frac{\beta}{\gamma}(1+\mu\|\Phi\|^2)Tr[(\widehat{\Theta}-\Xi)^T(\Lambda + \Lambda^T)(\widehat{\Theta}-\Xi)]\\
    &- 4 e_u^T P B_p (diag(\hat{\lambda})-diag(\lambda_{\epsilon}))\Delta u- \frac{\beta}{\gamma}(1 + \mu\|\Delta u\|^2)\\ &\cdot 2Tr[(diag(\hat{\lambda}) - diag(\lambda_{\epsilon}))^T(diag(\hat{\lambda}) - diag(\lambda_{\epsilon}))]\\
    &\leq -2\|e_u\|^2 + 4\|e_u\|\|P B_p \Lambda(\widehat{\Theta}-\Xi)\|_F \|\Phi\|\\
    &-\frac{2\beta}{\gamma}(1 + \frac{2\gamma}{\beta}\|PB\|_F^2\|\Phi\|^2) Tr\left[ (\widehat{\Theta}-\Xi)^T\Omega^T \Omega (\widehat{\Theta}-\Xi) \right]\\
    &+ 4\|e_u\|\|P B_p(diag(\hat{\lambda}) - diag(\lambda_{\epsilon}))\|_F\|\Delta u\|\\
    &-\frac{2\beta}{\gamma}(1 + \frac{2\gamma}{\beta}\|P B_p\|^2_F\|\Delta u\|^2)\\
    &\cdot Tr[(diag(\hat{\lambda}) - diag(\lambda_{\epsilon}))^T(diag(\hat{\lambda}) - diag(\lambda_{\epsilon}))] \\
    &= -\left[ \|e_u\| - 2\|P B_p\|_2 \|(\widehat{\Theta} - \Xi)^T\Omega\|_F\|\Phi\|\right]^2\\
    & -\left[ \|e_u\| - 2\|P B_p\|_2 \|diag(\hat{\lambda}) - diag(\lambda_{\epsilon})\|_F \right]^2 \leq 0
\end{align*}

Therefore, $e_u, \widehat{\Theta}_a, \Xi_a \in \mathcal{L}_{\infty}$. Define $K_{max} = \max \|\widetilde{\Theta}_a\|$,
\begin{equation*}
    K_{max} \geq \max \left( \sup\|\Tilde{K}_x\|, \sup\|\Tilde{K}_r\|, \sup\|\Tilde{\lambda}\| \right)
\end{equation*}

We define the following variables for compact notations in the error bound definition:
\begin{align*}
    q_{min} &= \min eig(Q), p_{min} = \min eig(P), p_{max} = \max eig(P),\\
    \rho &= \sqrt{\frac{p_{max}}{p_{min}}}, \Bar{u}_{min} = \min_i (u_{max, i}), \Bar{u}_{max} = \max_i (u_{max, i}),\\
    P_B &= \|P B_p \Lambda\|, \lambda_{min} = \min(eig(\Lambda))\\
\end{align*}
Additionally, the following variable definitions are used in the main stability result, Theorem~\ref{thm: stability}.
{\small
\begin{subequations}
    \begin{align*}
        \beta &= \frac{P_B K_{max}}{\|K_x^*\| + K_{max}}, a_0 = \frac{\Bar{u}_{min} K_{max}}{\|K_x^*\| + K_{max}}\\
        x_{min} &= \frac{3 P_B K_{max} (r_{max}+1) + 3P_B \|K_r^*\| r_{max}}{q_{min} - 3 P_B K_{max}}\\ 
        &+ \frac{2P_B \Bar{u}_{max}}{q_{min} - 3 P_B K_{max}}\\
        x_{max} &= \frac{P_B a_0}{|q_{min} - 2 P_B \|K_x^*\| |}\\
        \Bar{K}_{max} &= \frac{q_{min} - \frac{\rho}{a_0}\left( 3\|K_r^*\| r_{max} + 2\Bar{u}_{max} \right)|q_{min}}{3 P_B + \frac{3\rho}{a_0}(r_{max}+1)|q_{min} - 2P_B \|K_x^*\| | }\\
        &- \frac{2 P_B \|K_x^*\| |}{3 P_B + \frac{3\rho}{a_0}(r_{max}+1)|q_{min} - 2P_B \|K_x^*\| | }
    \end{align*}
\end{subequations}
} where all vector norms are 2-norm and the matrix norm in $P_B$ is the induced matrix norm, which implies $\|P B_p \Lambda x\| \leq P_B \|x\|$.

\begin{theorem}\label{thm: stability}
    For the closed-loop system with adaptive controller and adaptation laws described in \eqref{equ:plant} \eqref{equ:AC}, and \eqref{equ:HT_1}-\eqref{equ:HT_2}, $x(t)$ has bounded trajectories for $t \geq t_0$ if
    \begin{enumerate}
        \item $\|x(t_0)\| < \frac{x_{max}}{\rho}$
        \item $\sqrt{V(t_0)} < \Bar{K}_{max}\sqrt{\frac{\lambda_{min}}{\gamma_{max}}}$
    \end{enumerate}
    Moreover, $\|x(t)\| < x_{max}, \forall t \geq t_0$, and the error variable is of the same order as the difference between saturated input $B_{sat}(u(t))$ and the unsaturated $u(t)$, i.e., 
    \begin{equation*}
        \|e\| = \|x_p - x_m\| = O\left[ \sup_{\tau \leq t}\|\Delta u(\tau)\| \right]
    \end{equation*}
\end{theorem}

The proof is very similar to the one for Theorem~1 in \cite{schwager2005direct}. Note that in the proof especially parts related to the error bounds, adaptation laws are not directly involved, instead, parameter estimation error upper bound $K_{max}$ is used to bound the error term. The boundedness of $K_{max}$ relies on $\widetilde{\Theta}_a \in \mathcal{L}_{\infty}$, which is proved by showing the candidate in \eqref{equ:Lyapunov} is indeed a Lyapunov function. 

It should be noted that the stability results hold for any initial values of control gain matrix estimates $\hat{K}_1(t_0), \hat{K}_0(t_0)$. The main advantage of the MSAC is that it can provide stable solutions without any prior knowledge of unknown parameters, in comparison to policy iteration methods such as \cite{jiang2012computational} where an initial stable controller is assumed to be available, which requires at least partial information about the unknown parameters.

\subsubsection{Parameter learning}
Now that boundedness and asymptotic properties of the tracking error are established we proceed to learning of the unknown parameters. In order to learn $\Theta_a$, the regressor $\Phi_a$ must satisfy persistent excitation properties. We now introduce definitions of persistent excitation and the conditions for parameter learning.
\begin{definition}
    A bounded function $\Phi: [t_0, \infty) \to \mathbb{R}^{n_u}$ is persistently exciting (PE) if there exists $T > 0$ and $\alpha > 0$ such that
    \begin{equation}
        \int_t^{t+T} \Phi(\tau) \Phi^T(\tau) d\tau \geq \alpha I, \quad \forall t \geq t_0.
    \end{equation}
\end{definition}

If $\|\dot{\Phi}(t)\|$ is bounded for all $t$, equivalently, an alternative definition can be given as follows,
\begin{definition}
    $\Phi$ is PE if there exists an $\epsilon > 0$, a $t_2$ and a sub-interval $[t_2, t_2 + \delta_0] \subset [t, t+T]$ such that for all unit vectors $\omega \in \mathbb{R}^{n_u}$
    \begin{equation}\label{equ:PE}
        \frac{1}{T}\bigg\lvert\int_{t_2}^{t_2 + \delta_0} \Phi(\tau)^T \omega d \tau \bigg\rvert \geq \epsilon, \quad \forall t \geq t_0
    \end{equation}
\end{definition}

\begin{lemma}\label{lemma: theta_bound}
    Let $\epsilon$ and $\delta$ be given positive numbers. There exists $T = T(\epsilon, \delta)$ such that if $z(t) = [e_u(t)^T, \widetilde{\Theta}_a(t)^T]^T$ is a solution with $\|z(t_1)\| \leq \epsilon_1$, then there exists some $t_2 \in [t, t + T]$ such that $\|\widetilde{\Theta}_a(t_2)\|\leq \delta$.
\end{lemma}

\begin{theorem}\label{thm: PE}
    If $\Phi_a (t)$ satisfies the persistent excitation property in \eqref{equ:PE}, then the origin in \eqref{equ:HT_adap}, $e_u = 0, \Xi_a = 0, \Theta_a = 0$, is uniformly asymptotically stable.
\end{theorem}

Under the assumption that $\Phi_a (t)$ satisfies persistent excitation property, Lemma~\ref{lemma: theta_bound} and Theorem~\ref{thm: PE} state that arbitrary accuracy of estimated parameters can be achieved in finite time. Given a positive estimation accuracy $\delta$, we can always find $t_2$ such that $\|\widetilde{\Theta}_a\| \leq \delta$. This implies that given a $\delta$, there exists a large enough $T_{adap}$ such that, $\|\widetilde{\Theta}_a(T_{adap})\| \leq \delta$ is satisfied. We refer the readers to \cite{annaswamy2021online}, where detailed proofs can be found. As the underlying asymptotic convergence is not necessarily exponential, an explicit bound on $T_{adap}$ that leads to a desired accuracy $\delta$ is yet to be defined. Non-asymptotic tools \cite{sarker2023accurate} may need to be examined for this purpose.


\subsection{MPC: Optimize after parameter learning}
Now that the unknown parameter values have been estimated with accuracy $\delta > 0$, we can proceed to replace the unknown parameters in the optimal control problem considered in Section III with the parameter estimates. More precisely, at $t = t_0 + T_{adap}$, we compute the plant parameter estimates as:
\begin{equation}\label{equ:para_bound}
    \hat{A}_p = A_m - B_p \hat{\Lambda} \hat{K}_x(t_0 + T_{adap}), \hat{\Lambda} = diag(\hat{\lambda}(t_0 + T_{adap}))
\end{equation}

Since there is still some discrepancy between unknown parameters and their estimates, we choose a model predictive control (MPC) design to compute an approximate optimal control, $u_{MPC}$. This carried out at every sampling time instant $t_i$ as follows \cite{kouvaritakis2016model}.
\begin{subequations}\label{equ:LQT_MPC}
    \begin{align}
        \min_{u(\tau)} J &= h(x_p(t_i+T)) + \int_{t_i}^{t_i+T} l(x_p, u, \tau) d\tau\\
        \text{subject to } & x(t_0) = x_0 \\
        &\dot{x}_p(\tau) = \hat{A}_p x_p(\tau) + B_p \widehat{\Lambda} (u(\tau) ),  \forall \tau \in [t_i, t_i+T]\\
        & u(\tau)  \in U ,\quad \forall \tau \in [t_i, t_i+T]\\
        & U = \{u\in \mathbb{R}^{n_u}: \|u\| \leq u_{max}\}
    \end{align}
\end{subequations}
We denote $u^{AOC}(\tau), \tau \in [t_i, t_i+T]$ as the solution to \eqref{equ:LQT_MPC} and denote $u^{MPC}(\tau) = u^{AOC}(\tau), \forall \tau \in [t_i, t_{i+1}]$. We denote the corresponding cost-to-go as $V^{MPC}$. It is also assumed that $t_1 \geq t_0 + T_{adap}$.

Under Assumption~\ref{assum:R_diagonal}, by Corollary~\ref{coro: ball_cons_LQT}, $u^{MPC}$ is a linear feedback controller, denoted as $u^{MPC} = \hat{K}_1 x_p + \hat{K}_0$. It should be noted that the parameters of $u^{MPC}$, $\hat{K}_1$ and $\hat{K}_0$ are directly dependent on the estimated plant parameters $\hat{A}_p$ and $\Lambda$. Similarly, denoting the optimal solution of the problem in \eqref{equ:LQT_uncertain_1}-\eqref{equ:LQT_uncertain_2}
as $u^*$, it is easy to infer from Corollary~\ref{coro: ball_cons_LQT}, again, that $u^* = K_1^* x_p + K_0^*$.

Since $\|\widehat{\Theta}_a\|\leq \delta$, $\hat{\Theta}_a = [\hat{K}_x, \hat{K}_r, \widehat{\Lambda}]$, by \eqref{equ:para_bound},
\begin{equation*}
    \|\hat{A}_p - A_p\| = O(\delta), \quad \|\widehat{\Lambda} - \Lambda\| \leq \delta
\end{equation*}
Note that 
\begin{equation*}
    \|\hat{K}_1 - K_1^*\| = O(\delta), \|\hat{K}_0 - K_0^*\| = O(\delta)
\end{equation*}

Using Proposition~\ref{prop: dist_LQ_HJB}, it follows that the deviation in the corresponding cost-to-go functions is given by
\begin{equation}\label{equ:cost_gap}
    |V^{MPC} - V^*| \leq O(\delta^2)
\end{equation}

In conclusion, the discrepancy between the cost associated with system trajectories generated using $u^{MPC}$ and the optimal cost is quantified in \eqref{equ:cost_gap}, and it summarizes the main advantage of the proposed MSAC-MPC controller. That is, the MSAC-MPC controller results in a cost that differs from the optimal cost by an order of magnitude comparable to parameter error $\delta$ and is of the order $\delta^2$. This optimal cost is evaluated after a time $T_{adap}$ elapses after $t_0$. The benefit of the proposed method is the reduction of the computational burden over $[t_0, t_0+T_{adap}]$ in comparison to \cite{adetola2009adaptive} where a min-max optimization will have to be solved to compute a controller robust for all values of $A_p$ and $\Lambda$ in a compact set. Instead, our controller is chosen to be stable and sub-optimal over $[t_0, t_0+T_{adap}]$ and near-optimal for $t\geq t_1 > t_0+T_{adapt}$. Ensuring optimality over [$t_0,t_0+T_{adap}]$ is a difficult problem that remains to be addressed.

\section{Numerical Example}
A numerical example is implemented to demonstrate the ability of proposed MSAC-MPC algorithm to ensure closed-loop system stability, learn unknown parameter values before $T_{adap}$, and perform optimal control after $T_{adap}$. The plant is chosen as an unstable, second-order linear system with two inputs in the form of \eqref{equ:plant}, where
\begin{equation*}
    A_p = \begin{bmatrix}
        1 & 1\\ 0 & 1
    \end{bmatrix},
    B_p = \begin{bmatrix}
        1 & 0 \\ 0 & 1
    \end{bmatrix}, \Lambda = \begin{bmatrix}
        1 & 0 \\ 0 & 1
    \end{bmatrix}
\end{equation*}

The reference system is chosen as specified in \eqref{equ:ref_sys} with 
\begin{equation*}
    A_m = \begin{bmatrix}
        -1 & 1\\ 0 & -2
    \end{bmatrix}, B_m = \begin{bmatrix}
        1 & 0 \\ 0 & 1
    \end{bmatrix}
\end{equation*}

The input magnitude saturation function $B_{sat}(u)$ is defined in \eqref{equ:input_sat}, and the admissible input set $U$ is defined in \eqref{equ:LQT_uncertain_2}, where the $u_{max}$ is chosen as $u_{max} = 8$. Without loss of generality, it assumed that $t_0 = 0$.

The exogenous signal desired to be tracked is defined as,
\begin{equation*}
    x_d(t) = \begin{bmatrix}
        sin(t) + sin(3t) + sin(5t) + sin(7t)\\
        sin(2t) + sin(4t) + sin(6t)
    \end{bmatrix}
\end{equation*}
Note that there are 12 uncertain parameters in $\Theta_a$ to be estimated. Choosing the above $x_d$ with sinusoidal signals at 7 different frequencies ensures $\Phi_a$ to have persistent excitation properties in \eqref{equ:PE} \cite{narendra2012stable}.

The penalty matrix $Q$, $R$ defined in \eqref{equ:LQT_uncertain_1} are chosen as follows,
\begin{equation*}
    Q = \begin{bmatrix}
        20 & 0 \\ 0 & 20
    \end{bmatrix}, R = \begin{bmatrix}
        1 & 0 \\ 0 & 1
    \end{bmatrix}
\end{equation*}

The initial condition of the state $x_p$ is set to $x_p(0) = [0, 0]^T$. The initial parameter estimator values are chosen as $\widehat{\Theta}_a(0) = 0.8\Theta_a$.

In the simulation, $T_{adap} = 32\pi$ is chosen to ensure parameter learning is performed sufficiently. MSAC controller defined in \eqref{equ:AC}, \eqref{equ:feedback_r}, \eqref{equ:e_u_dyn}, \eqref{equ:HT_1}-\eqref{equ:HT_2} is applied to \eqref{equ:plant} during $t \in [0, T_{adap}]$. The MPC controller $u^{MPC}$ is run during $[T_{adap}, T_{adap} + T_{MPC}]$, where $T_{MPC} = 8$ [sec]. Simulation results for MSAC controller is shown in Figure~\ref{fig:MSAC_x}-\ref{fig:MSAC_theta}.

\begin{figure}[H]
    \begin{center}
        \begin{picture}(220, 80)
            \put(0, -10){\epsfig{file=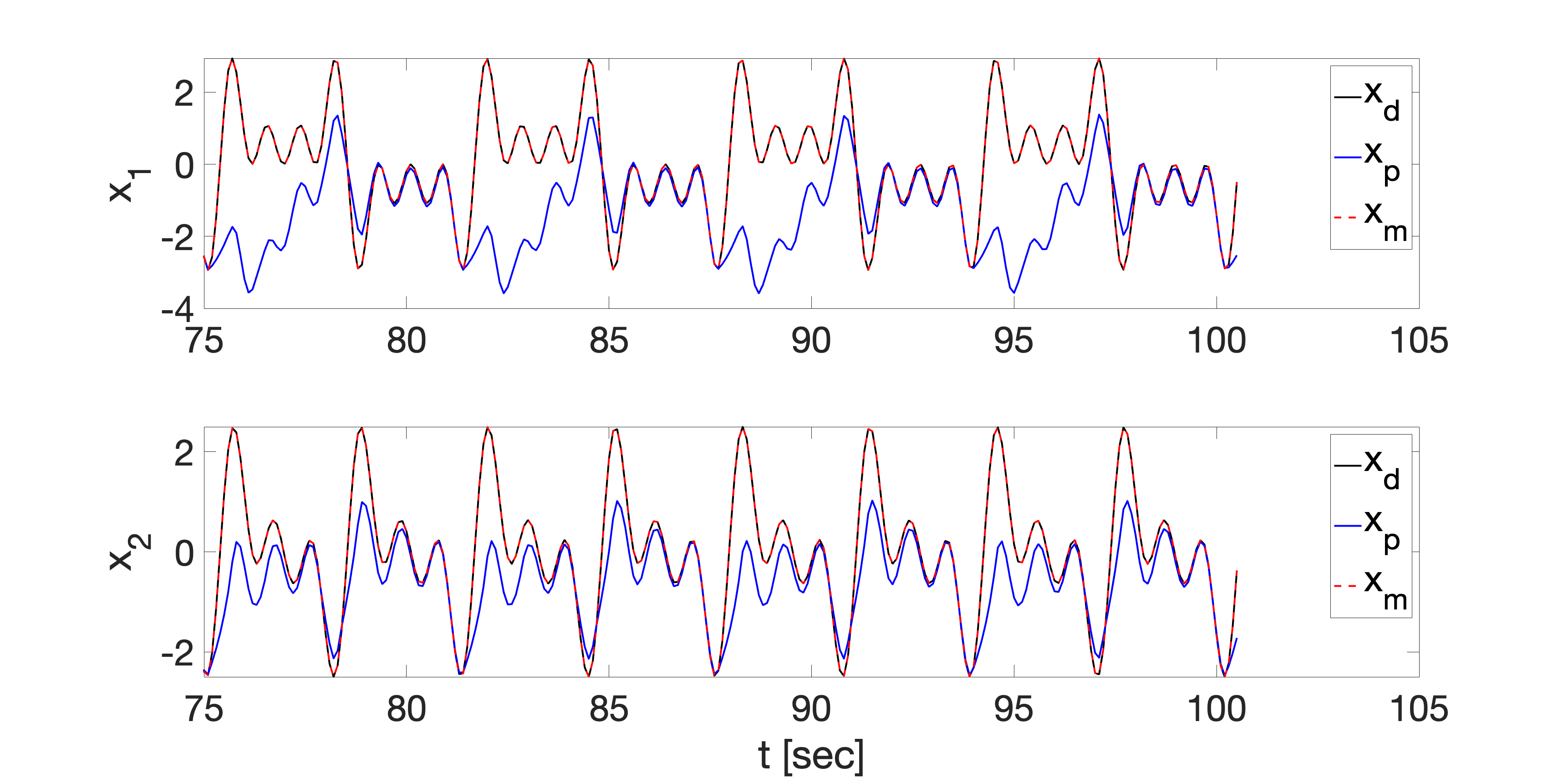, height=1.4in}}
        \end{picture}
    \end{center}
    \caption{State trajectories of plant $x_p$, reference system $x_m$, and exogenous signal $x_d$.}
    \label{fig:MSAC_x}
\end{figure}
\begin{figure}[H]
    \begin{center}
        \begin{picture}(220, 75)
            \put(0, -10){\epsfig{file=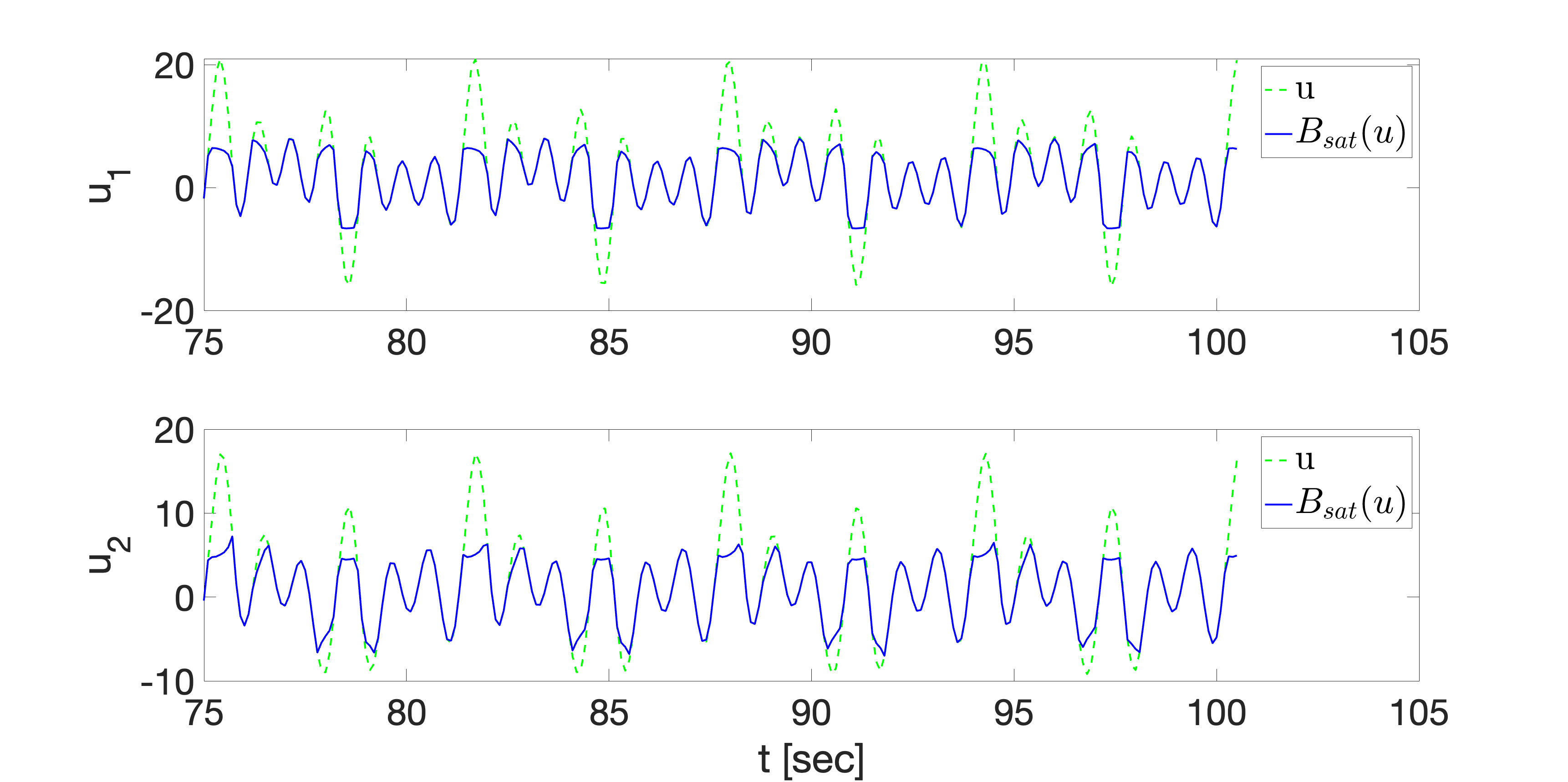, height=1.4in}}
        \end{picture}
    \end{center}
    \caption{Time histories of unsaturated $u$ and saturated $B_{sat}(u)$.}
    \label{fig:MSAC_u}
\end{figure}
\begin{figure}[!h]
    \begin{center}
        \begin{picture}(240, 25)
            \put(0, -10){\epsfig{file=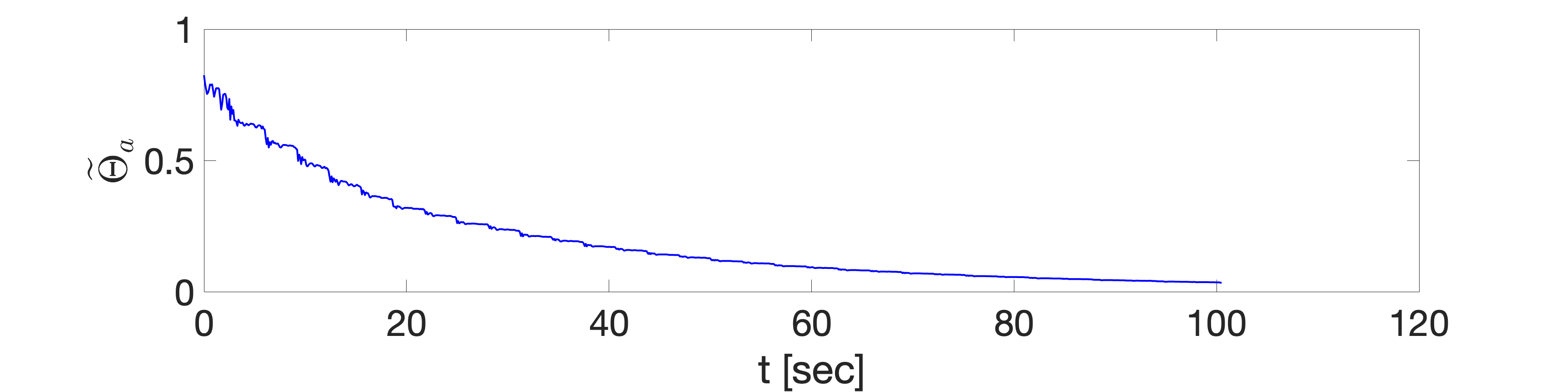, height=0.75in}}
        \end{picture}
    \end{center}
    \caption{2-norm of parameter estimation error $\widetilde{\Theta}_a$.}
    \label{fig:MSAC_theta}
\end{figure}

In Figure~\ref{fig:MSAC_x}, it is observed that $x_m$ follow $x_d$ closely. Due to active control input saturation \eqref{equ:input_sat} observed in Figure~\ref{fig:MSAC_u}, the tracking error $\|x_p - x_d\|$ is not converging to zero, but it is the same order of magnitude as $\Delta u$ defined in defined in Section IV.A, which demonstrates the MSAC controller's ability to provide stability guarantees. It is also observed in Figure~\ref{fig:MSAC_theta} that under PE condition, the MSAC controller is able to reduce the parameter estimation error rapidly, as stated in Theorem~\ref{thm: PE}.

After $T_{adap}$, the MSAC-MPC controller switches to $u^{MPC}$ defined in Section IV.B. The simulation results are shown in Figure~\ref{fig:MPC_x}-\ref{fig:MPC_u}, where $x_p$ is the plant state trajectory generated by applying $u^{MPC}$ to \eqref{equ:plant}, and $x_p^*$ is the latent optimal state trajectory generated by applying $u^*$ to \eqref{equ:plant}. It is observed that $x_p$ and $u^{MPC}$ match with $x_p^*$ and $u^*$ closely, respectively.
\begin{figure}[!h]
    \begin{center}
        \begin{picture}(200, 100)
            \put(-10, 0){\epsfig{file=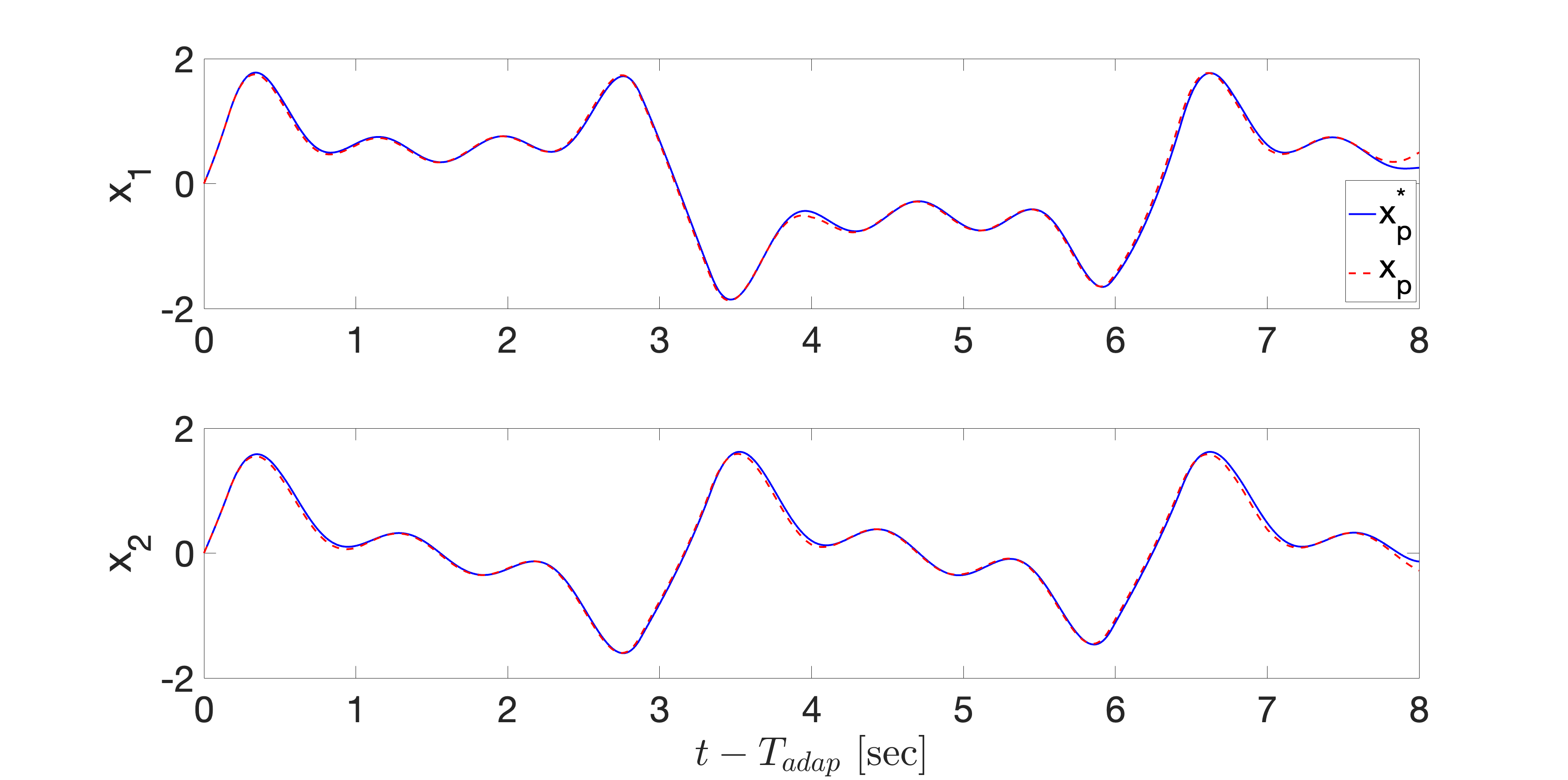, height=1.4in}}
        \end{picture}
    \end{center}
    \caption{Comparison between time histories of plant state $x_p$ and optimal state $x_p^*$.}
    \label{fig:MPC_x}
\end{figure}
\begin{figure}[!h]
    \begin{center}
        \begin{picture}(200, 100)
            \put(-10, 0){\epsfig{file=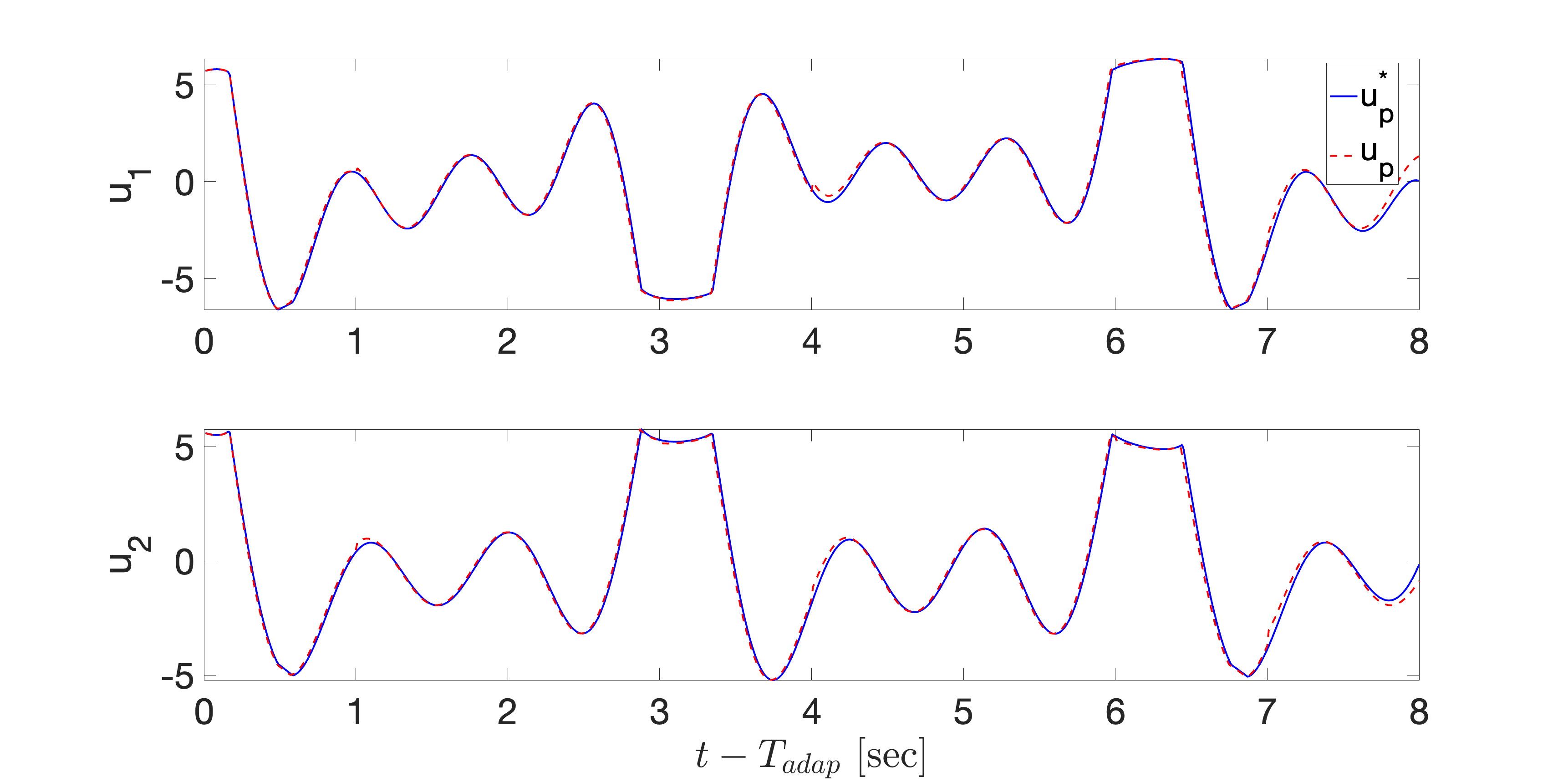, height=1.4in}}
        \end{picture}
    \end{center}
    \caption{Comparison between time histories of $u^{MPC}$ and $u^*$.}
    \label{fig:MPC_u}
\end{figure}

\section{Conclusions}
In this paper, we have developed a MSAC-MPC controller to address linear quadratic tracking optimal control problems under parametric uncertainties and input saturation. Theoretical results have been developed to show the proposed MSAC controller is able to provide stable solutions and achieve parameter learning under persistent excitation, and that after switching from MSAC to MPC controller, the optimality gap is proportional to the parameter estimation error bound. A numerical example based on an unstable second-order, two-input system has been developed to demonstrate the effectiveness of the MSAC-MPC controller along with its parameter learning ability and well-defined optimality gap. We note that during the first step when adaptation is being carried out, the controller is stable but not optimal, and that during the second step, it is near-optimal. How these optimality gaps can be further reduced is a topic for future research.

\addtolength{\textheight}{-12cm}   






\bibliographystyle{IEEEtran}
\bibliography{ref.bib}

\end{document}